\documentclass{amsart}
\usepackage{amsmath,amsfonts,amssymb,amscd,amsthm,epsf}
\newtheorem{prop}{Proposition}[section]
\newtheorem{thm}[prop]{Theorem}
\newtheorem{Ethm}[prop]{Eliashberg's Theorem}
\newtheorem{cor}[prop]{Corollary}

\newtheorem{lem}[prop]{Lemma}

\theoremstyle{definition}

\newtheorem{example}[prop]{Example}
\newtheorem{examples}[prop]{Examples}
\theoremstyle{remark}
\def\complex{{\mathbb C}}
\def\C{{\mathbb C}}
\def\CP{{\mathbb C \mathbb P}}
\def\zed{{\mathbb Z}}
\def\real{{\mathbb R}}
\def\R{{\mathbb R}}

\def\int{\mathop{\rm int}\nolimits}

\def\dim{\mathop{\rm dim}\nolimits}
\def\rel{\mathop{\rm rel}\nolimits}

\begin{document}
\title{Constructing Stein manifolds after Eliashberg}
\author{Robert E. Gompf}
\thanks{Partially supported by NSF grant DMS-0603958.}
\address{Department of Mathematics, The University of Texas at Austin,
1 University Station C1200, Austin, TX 78712-0257}
\email{gompf@math.utexas.edu}
\date{October 9, 2008}
\begin{abstract}
A unified summary is given of the existence theory of Stein manifolds in all dimensions, based on published and pending literature. Eliashberg's characterization of manifolds admitting Stein structures requires an extra delicate hypothesis in complex dimension 2, which can be eliminated by passing to the topological setting and invoking Freedman theory. The situation is quite similar if one asks which open subsets of a fixed complex manifold can be made Stein by an isotopy. As an application of these theorems, one can construct uncountably many diffeomorphism types of exotic $\R^4$s realized as Stein open subsets of $\C^2$ (i.e. domains of holomorphy). More generally, every domain of holomorphy in $\C^2$ is topologically isotopic to other such domains realizing uncountably many diffeomorphism types. Any tame $n$-complex in a complex $n$-manifold can be isotoped to become a nested intersection of Stein open subsets, provided the isotopy is topological when $n=2$. In the latter case, the Stein neighborhoods are homeomorphic, but frequently realize uncountably many diffeomorphism types. It is also proved that every exhausting Morse function can be subdivided to yield a locally finite handlebody of the same maximal index, both in the context of smooth $n$-manifolds and for Stein surfaces.
\end{abstract}
\maketitle


\section{Introduction}

One of Eliashberg's foundational achievements is his pioneering work on the existence theory of Stein manifolds. These manifolds have been studied by complex analysts for most of the past century, and so have many equivalent definitions. The definition that most immediately indicates their fundamental nature is that they are complex affine analytic varieties, i.e., complex manifolds that embed holomorphically as closed subsets of $\C^N$. The importance of Stein manifolds mandates the study of two basic existence questions. Most obviously, we should ask which abstract smooth manifolds admit Stein structures. However, we can also work inside a fixed ambient complex manifold $X$. Every open subset of $X$ is itself a complex manifold, so we can ask which open subsets of $X$ are Stein. The study of such Stein open subsets has had a long and continuing history, already in the special case $X=\C^n$. For example, if $X$ is $\C^n$ (or any other Stein manifold), then its Stein open subsets $U$ have been characterized as being {\em domains of holomorphy\/}, meaning that they satisfy a certain maximality condition for extending holomorphic functions. (Specifically, for every connected, open $V\subset X$ extending outside of $U$, and every component $V_0$ of $U\cap V$, there are holomorphic functions on $U$ whose restriction to $V_0$ does not extend to $V$.) Since the Stein condition can be destroyed by a tiny perturbation of $U$, we reformulate the ambient existence question to put it within reach of a topological answer: Which open subsets of $X$ can be made Stein by an {\em isotopy\/}, i.e., a homotopy of the inclusion map through embeddings? (We define embeddings to be 1-1 immersions, but do not require them to be proper, so we are asking when an open set can be deformed into a new open set whose complex structure inherited from $X$ is Stein.) Building on Eliashberg's work of the late 1980s, we can now completely answer both the abstract and ambient existence questions. In complex dimension $n\ne 2$, Eliashberg completely characterized those manifolds admitting Stein structures by a simple differential topological condition. A similar statement can be obtained in the ambient setting. When $n=2$, one obtains the corresponding characterizations by imposing additional delicate conditions. Alternatively, one can eliminate these extra conditions by passing to the topological category, i.e., working up to homeomorphism or up to isotopy through topological ($C^0$ rather than $C^\infty$) embeddings, and invoking Freedman theory. As an application, one obtains domains of holomorphy in $\complex^2$ realizing uncountably many diffeomorphism types of exotic $\real^4$s. More generally, every domain of holomorphy in $\C^2$ is topologically isotopic to uncountably many other diffeomorphism types of such domains (all of which are homeomorphic).

This paper surveys published and pending literature to give a unified view across all dimensions of these consequences of Eliashberg's work. The abstract and ambient results are presented in Sections~\ref{Abstract} and \ref{Ambient}, respectively, and their proofs are sketched in Section~\ref{proofs}. Section~\ref{Onion} reexamines the ambient theory from several other viewpoints, replacing isotopy by the sharper notion of ambient isotopy and discussing when embedded CW-complexes can be described (after isotopy) as nested intersections of Stein open subsets. In the 4-dimensional topological setting, this frequently yields an uncountable neighborhood system realizing uncountably many exotic smooth structures. Section~\ref{Appendix} is an appendix proving several lemmas that are useful in earlier sections.

\section{Abstract characterization} \label{Abstract}

To motivate Eliashberg's work, we first consider necessary conditions for the existence of a Stein structure. Since a Stein manifold $U$ is complex, it comes endowed with an almost-complex structure, that is, its tangent bundle $TU$ is a complex vector bundle. We capture this structure with the bundle automorphism $J$ of $TU$ given by fiberwise multiplication by $i$. If such an almost-complex structure $J$ comes from a complex structure, then it uniquely determines the latter. However, we usually consider almost-complex structures only up to homotopy. The homotopy classification of complex bundle structures on $TU$ (or on any real vector bundle) is a standard problem in algebraic topology. Thus, it is reasonable to refine our original abstract existence question by asking which almost-complex structures on a given manifold (if any exist) are homotopic to Stein structures. In addition to having an almost-complex structure, a Stein manifold $U$ satisfies another classical condition: It admits an exhausting Morse function (i.e. a smooth, proper map $\varphi:U\to [0,\infty)$ whose critical points are nondegenerate) such that each critical point has index $\le n=\dim_\C U$. (Thus, $U$ has real dimension $2n$ but the homotopy type of an $n$-complex.) Surprisingly, Eliashberg's Theorem asserts that these necessary conditions are also sufficient.

\begin{Ethm}\label{absn} \cite{E}.
For $n\ne 2$, let $U$ be a $2n$-manifold with an almost-complex structure $J$. Then $J$ is homotopic to a Stein structure if and only if $U$ admits an exhausting Morse function whose critical points all have index $\le n$.
\end{Ethm}

\noindent In particular, $U$ admits a Stein structure if and only if it admits an almost-complex structure and suitable Morse function. The original proof appeared in \cite{E}, but a simpler and more expository version will be given in the book \cite{CE}. We will discuss the latter version in Section~\ref{proofs}.

It is sometimes useful to replace Morse theory by the language of handlebodies (e.g. \cite{GS}). Recall that a {\em handle of index $k$\/}, or {\em $k$-handle\/}, is a copy of $D^k\times D^{m-k}$ attached to the boundary of an $m$-manifold along its attaching region $\partial D^k\times D^{m-k}$. A (topological or smooth) {\em handlebody\/} $H$ is a manifold built from the empty set by successively attaching handles (necessarily beginning with a 0-handle). In the smooth case, we glue each handle by a smooth embedding of its attaching region, then smooth the resulting corners. (These corners occur along $\partial D^k\times\partial  D^{m-k}$. By uniqueness of tubular neighborhoods, we can canonically identify a neighborhood of the corner locus with the product of $\partial D^k\times\partial  D^{m-k}$ with $\R^2$ minus the open first quadrant, then smooth using the essentially unique smoothing of the last factor.) Handlebodies are necessarily {\em locally finite\/} in the sense that only finitely many handles attach to a given handle. (Attaching infinitely many handles to a compact boundary region would cause clustering destroying the manifold structure.) We also require handlebodies to be {\em Morse ordered\/}, so each $k$-handle attaches to a subhandlebody consisting of handles with index $<k$. For compact handlebodies (i.e. those with finitely many handles), this costs no generality by a general position argument, but infinite handlebodies are restricted, e.g. requiring infinitely many 0-handles. We impose this constraint since the cores $D^k\times \{ 0\}$ of the handles will then fit together inside the handlebody to give an embedded CW-complex $K$. (Note that a CW-complex is Morse-ordered by definition and locally finite if it embeds in a manifold.) If there are no $m$-handles, then $H-K$ is homeomorphic (diffeomorphic in the smooth case) to $\partial H\times (0,1]$, realizing $H$ as the mapping cylinder of a map $\partial H\to K$.

To relate Morse functions with handlebodies, suppose $a\in [0,\infty)$ is a critical value of an exhausting Morse function $\varphi:U\to \R$, corresponding to a unique critical point $p$ in $\varphi^{-1}[a-\epsilon,a+\epsilon]\subset U$. Then $\varphi^{-1}[0,a+\epsilon]$ is obtained from $\varphi^{-1}[0,a-\epsilon]$ by adding a collar $I\times \varphi^{-1}(a-\epsilon)$ to the boundary and then attaching a handle with the same index as $p$ (e.g.  \cite{M}). If there are only finitely many critical points, we may then identify $U$ with the interior of a handlebody by collapsing out the intervening collars. In our situation, though, there may be infinitely many critical points (although only finitely many in each $\varphi^{-1}[0,b]$ since $\varphi$ is proper). Collapsing the collars may no longer yield a Morse-ordered handlebody. One could remedy this by simply adding collars into the definition of a handlebody, but only at the expense of losing Morse-ordering or local finiteness. This would complicate the statements of some of our theorems, and lose the natural CW cores and mapping cylinder structure. Instead, we show in the Appendix (Lemma~\ref{decomp}) that suitable subdivision allows us to identify $U$ with the interior of a handlebody, even if $\varphi$ has infinitely many critical points, and the maximal index of the handles is that of the critical points of $\varphi$. Applying this to Theorem~\ref{absn}, we conclude:

\begin{cor}\label{absn1}
Under the hypotheses of Eliashberg's Theorem, $J$ is homotopic to a Stein structure if and only if $U$ is diffeomorphic to the interior of a handlebody whose handles all have index $\le n$.
\end{cor}

As we will see in Section~\ref{proofs}, Eliashberg constructs his Stein structure from $J$ by induction on increasing critical values, or from the other viewpoint, by induction on handles. The induction step works in all cases except for 4-dimensional 2-handles, resulting in a proof for all $n\ne2$. In the problematic case, the induction step may or may not work, depending on the framing of the 2-handle, so the following theorem results for Stein {\em surfaces\/} (i.e. the $n=2$ case).

\begin{thm}\label{abs2} (Eliashberg.)
An oriented 4-manifold admits a Stein structure if and only if it is diffeomorphic to the interior of a handlebody whose handles all have index $\le 2$, and for which each 2-handle is attached along a Legendrian knot (in the standard contact structure on the relevant boundary 3-manifold) with framing obtained from the contact framing by adding one left twist.
\end{thm}

\noindent For terminology and additional explanation, see Section~\ref{proofs}. The main observation there is that each 2-handle is attached to a finite union of 0- and 1-handles, whose boundary $\#m S^1\times S^2$ admits essentially a unique tight contact structure, with respect to which we can define Legendrian knots and contact framings. Theorem~\ref{abs2} was essentially known to Eliashberg when he wrote \cite{E}, but it was not contemporaneously published. We have sharpened it slightly, using the Appendix (Lemma~\ref{steindec}) to write it in the context of Morse-ordered handlebodies without collars. (This sharpening is what makes a concise statement possible when there are infinitely many handles. From the Morse viewpoint, one must build the Stein structure on the possibly complicated manifold below the level of a new 2-handle before the framing condition makes sense.) For further discussion and applications of the theorem, see e.g. \cite{CE}, \cite{Ann}, \cite{GS} or \cite{OS}. Note that a handlebody with all indices $\le2$ has the homotopy type of a 2-complex, so if it is oriented, it automatically admits almost-complex structures (respecting the given orientation). Thus it suffices to hypothesize an orientation rather than an almost-complex structure. On the other hand, the almost-complex structure is completely determined by the Legendrian link along which the 2-handles are attached (see Section~\ref{proofs}). Unlike in other dimensions, there are manifolds admitting Stein structures realizing some, but not all, homotopy classes of almost-complex structures. (In fact, this occurs for every Stein surface with $H_2(U;\zed)\ne0$, by the Adjunction Inequality from Seiberg-Witten theory \cite{LM}.)

We now have a complete characterization, in all dimensions, of which manifolds admit Stein structures, and which homotopy classes of almost-complex structures are so realized. However, the criterion for 4-manifolds is much harder to apply than in other dimensions. Furthermore, if we ignore the delicate framing condition, we immediately run into trouble. The simplest handlebody interior requiring a 2-handle is $S^2\times \real^2$. While this admits a complex structure ($\CP^1\times \C$), it admits no Stein structure, because it contains a homologically essential 2-sphere with trivial normal bundle. (By Seiberg-Witten theory, every smoothly embedded essential sphere in a Stein surface has normal Euler number $\le -2$ \cite{LM}.) To simplify the 4-dimensional characterization, we apply the fundamental principle of Freedman \cite{F}, \cite{FQ}, that high-dimensional differential topology frequently works in dimension 4 after we give up smoothness and enter the purely topological world. We replace the problematic 2-handles by more flexible {\em Casson handles\/}, which are homeomorphic to the open 2-handle $D^2\times \real^2$ but have exotic smooth structures. We then recover a precise analog of Eliashberg's Theorem up to homeomorphism.

\begin{thm}\label{abstop} \cite{Ann}.
An oriented 4-manifold is homeomorphic to a Stein surface if and only if it is homeomorphic to the interior of a handlebody whose handles all have index $\le 2$. Every homotopy class of almost-complex structures on such a handlebody interior is realized by an orientation-preserving homeomorphism to a Stein surface.
\end{thm}

\noindent Note that 4-dimensional topological handlebodies are uniquely smoothable, since they are made by gluing along 3-manifolds and the gluing maps are uniquely smoothable. Thus, it does not matter whether the handlebody (or 4-manifold) in the theorem is topological or smooth. To see why a homeomorphism determines a bijection of homotopy classes of almost-complex structures (which can be reinterpreted as spin$^\complex$-structures since the base is essentially a 2-complex), see \cite{Ann}. One way to interpret this result is that Eliashberg's Theorem holds in all dimensions, provided that when $n=2$ we are allowed to change the smooth structure on the manifold. The resulting smooth structures have peculiarities typical to dimension 4. Such Stein surfaces are typically not diffeomorphic to the interior of any compact 4-manifold, even if the original handlebody is finite. Furthermore, a single handlebody often corresponds to uncountably many diffeomorphism types of Stein surfaces.

\begin{examples} \label{abstract}
a) According to this theorem, $S^2\times \real^2$ is homeomorphic to Stein surfaces realizing all possible homotopy classes of almost-complex structures. (These homotopy classes are classified by their first Chern number, which can be any even integer.) As we have seen, such a  Stein surface cannot be diffeomorphic to $S^2\times \real^2$. In fact, there are uncountably many diffeomorphism types of Stein surfaces $U$ homeomorphic to $S^2\times \real^2$, and we can arrange the minimal genus of a smoothly embedded surface generating $H_2(U)$ to be any preassigned positive integer \cite{steintop}. This is the simplest nontrivial example, yet it is rather typical (except that the Chern class does not completely classify almost-complex structures when there is 2-torsion in $H^2(U;\zed)$).

b) One of the most peculiar phenomena unique to 4-dimensional topology is the existence of exotic $\R^4$s, that is, manifolds homeomorphic to Euclidean space $\R^4$ but not diffeomorphic to it. These have a complicated history, dating back to ideas of Casson in the 1970s \cite{C}. Most directly relevant to this paper is \cite{DF}, which exhibits uncountably many diffeomorphism types of exotic $\R^4$s, each occurring as an open subset of $\R^4$ (with its usual smooth structure). After simplifying (\cite{BG}, see also \cite{GS}), one can explicitly build each member of such a family with a 0-handle, two 1-handles, a 2-handle and a Casson handle (taking the resulting interior). By a careful application of Theorem~\ref{abs2}, using the main idea behind Theorem~\ref{abstop} to deal with the Casson handle, one can endow such an uncountable family of exotic $\R^4$s with Stein structures \cite{Ann}. In contrast, there is an uncountable family of exotic $\R^4$s that cannot be smoothly embedded in $\R^4$ \cite{Ta}, and these cannot be constructed as handlebody interiors without using infinitely many 3-handles \cite{T}, so they cannot admit Stein structures.
\end{examples}

\section{Ambient characterization} \label{Ambient}

We now return to the question of which open subsets $U$ of a fixed complex surface $X$ are isotopic to Stein open subsets. We wish to find characterizations analogous to the theorems of the previous section. Note that $U$ inherits a complex structure from $X$, and any isotopy of $U$ induces a homotopy of complex structures (which we now view as almost-complex structures). Thus, we should ignore our previous hypotheses regarding almost-complex structures, since they are now built into the topological setup. We are left with hypotheses regarding Morse functions or handlebodies, and these turn out also to suffice in the ambient setting.

We first state the analog of Theorem~\ref{absn} and Corollary~\ref{absn1}.

\begin{thm}\label{ambin} (Eliashberg.)
For $n\ne2$, an open subset $U$ of a complex $n$-manifold $X$ is smoothly isotopic to a Stein open subset if and only if it admits an exhausting Morse function whose critical points all have index $\le n$, or equivalently, it is diffeomorphic to the interior of a smooth handlebody with all indices $\le n$.
\end{thm}

\noindent Our requirement that handlebodies be Morse-ordered and without collars is again convenient here, for we can think of embedded handlebodies as tubular neighborhoods of smoothly embedded CW-complexes. We investigate this correspondence more deeply in Section~\ref{Onion}. Theorem~\ref{ambin} is implicit in \cite{E} in the sense that it follows by the method of proof of Theorem~1.3.6 of that paper. However, it does not seem to have previously appeared in the above form. There is essentially a complete proof in a paper of Forstneri{\v c} and Slapar \cite{FoS}. Their main objective was to extend Eliashberg's work to make a manifold Stein while a preassigned continuous map to some complex manifold becomes homotopic to a holomorphic map. If one erases the parts concerning the auxiliary map, however, one essentially obtains the proof of Theorem~\ref{ambin} implicit in \cite{E}. (More precisely, one immediately obtains Theorem~\ref{ambin} by applying Theorem~1.2 of \cite{FoS} to the complex manifold $U$, setting the auxiliary complex manifold equal to a point.)

Next we consider the ambient analog of Theorem~\ref{abs2}. We have already seen that Theorem~\ref{ambin} fails when $n=2$. (Consider the obvious embedding $S^2\times\R^2\subset \C^2$.) The problem is that the manifold $U$ may not even abstractly admit a Stein structure. The following theorem bypasses this difficulty.

\begin{thm} \label{ambi2} \cite{steindiff}.
An open subset $U$ of a complex surface is smoothly isotopic to a Stein open subset if and only if the induced complex structure on $U$ is homotopic (through almost-complex structures on $U$) to a Stein structure on $U$.
\end{thm}

\noindent Thus, the difficulty of making open subsets Stein by smooth isotopy in this dimension is already contained in the problem of abstractly creating a Stein structure, and is addressed by Theorem~\ref{abs2}.

\begin{example} \label{R4}
Let $U\subset \R^4=\C^2$ be one of the exotic $\R^4$s of Example~\ref{abstract}(b). We have already seen that $U$ admits a Stein structure. But $U$ is contractible, so it has a unique homotopy class of almost-complex structures compatible with the given orientation. By arranging the identification of $\R^4$ with $\C^2$ to respect this orientation, we guarantee that the induced complex structure on $U$ is homotopic to a Stein structure. Since a smooth isotopy preserves the diffeomorphism type of the domain, Theorem~\ref{ambi2} exhibits a domain of holomorphy in $\C^2$ diffeomorphic to $U$, an exotic $\R^4$. According to Example~\ref{abstract}(b), we obtain uncountably many diffeomorphism types of such exotic $\R^4$ domains of holomorphy in $\C^2$. More generally, we can replace $\C^2$ in this argument by an $\epsilon$-ball, obtaining a similar uncountable family inside any complex surface. The same method can be used to find many contractible domains of holomorphy in $\C^2$ that (unlike exotic $\R^4$s) are bounded by smooth, compact, pseudoconvex homology 3-spheres \cite{steindiff}.
\end{example}

Finally, we turn to the analog of Theorem~\ref{abstop}. As in that theorem, the extra hypothesis disappears when we replace 2-handles by Casson handles. We obtain a precise analog of Theorem~\ref{ambin} for $n=2$, up to topological isotopy.

\begin{thm}\label{ambitop} \cite{JSG}, \cite{steintop}.
An open subset $U$ of a complex surface $X$ is topologically isotopic to a Stein open subset if and only if it is homeomorphic to the interior of a handlebody with all indices $\le 2$.
\end{thm}

\noindent This is essentially proven, in an expository manner, in \cite{JSG}. It is also explicitly given in \cite{steintop} as a corollary of a stronger theorem (essentially Theorem~\ref{handle}(c) below). It can be derived from Theorem~\ref{ambi2} and a strengthened form of Theorem~\ref{abstop}: The Stein surface associated by the latter theorem to the above $U$ can actually be assumed to smoothly embed in $U$, so that the embedding is topologically isotopic to the identity map on $U$, and Theorem~\ref{ambi2} immediately gives the required Stein open subset. Alternatively, Theorem~\ref{ambitop} can be used to prove Theorem~\ref{abstop}: The handlebody interior of the latter theorem can be easily given a complex structure homotopic to the given almost-complex structure. (For example, construct a suitable continuous map from the core 2-complex to a rational surface, then homotope it to an immersion and pull back the complex structure.) Theorem~\ref{ambitop} with $U=X$ completes the proof.

\begin{examples}
a) Any given smooth embedding $S^2\to \C^2$ has a tubular neighborhood diffeomorphic to $S^2\times\R^2$. This is topologically (but obviously not smoothly) isotopic to a Stein open subset. As in Example~\ref{abstract}(a), we can realize any positive integer $g$ as the minimal genus of the generator of such a domain of holomorphy (for any fixed isotopy class of embeddings of $S^2$). For more on Stein neighborhoods of isotoped embedded surfaces, see \cite{JSG}, \cite{steintop}.

b) More generally, let $U\subset\C^2$ be any open subset homeomorphic to the interior of a handlebody with all indices $\le 2$. Then $U$ is topologically isotopic to a Stein open subset $V$. By Example~\ref{R4}, there are uncountably many diffeomorphism types of Stein exotic $\R^4$s  in $\C^2-V$. We can connect any one of these to $V$ by a 1-handle in $\C^2$, preserving the Stein condition. (This requires a small amount of control of the boundaries of these subsets, but is not a problem here.) By the method of \cite{DF}, we can then distinguish uncountably many diffeomorphism types of these domains of holomorphy, all topologically isotopic (hence homeomorphic) to $U$ \cite{steintop}. Thus, if $U\subset\C^2$ is topologically isotopic to any domain of holomorphy, then there are uncountably many diffeomorphism types of such. Note that this construction preserves the minimal genus of every homology class of $V$, so in (a) above we have uncountably many diffeomorphism types for each embedding of $S^2$ and each $g>0$. One would expect this uncountability to be typical, with any complex surface $X$ in place of $\C^2$. This can be proven in many special cases (\cite{steintop}, see also Theorem~\ref{onion} below), but the general case seems out of range of current technology.
\end{examples}

\section{Eliashberg's method and 4-dimensional 2-handles} \label{proofs}

The theorems of the previous two sections can be thought of as manifestations of Gromov's {\em h-principle} \cite{Gr}, \cite{EM}. For many problems such as constructing complex or symplectic structures or smooth immersions, one first has to solve a topological problem at the level of bundle theory. The h-principle for such a problem asserts that a solution at the level of bundle theory guarantees a solution to the complete problem. This clearly fails for putting complex structures on closed manifolds: A solution at the level of bundles is an almost-complex structure, but many closed, almost-complex manifolds do not admit complex structures, and there seems to be no hope of reducing the gap to a topological criterion. In contrast, we have seen that the existence of Stein structures and domains of holomorphy are determined by topological criteria. We will now sketch the proofs of the theorems of the two previous sections, emphasizing the underlying topology for the h-principle, which is the source of the difficulties surrounding 4-dimensional 2-handles.

\begin{proof}[Proof of Eliashberg's Theorem~\ref{absn}]
We sketch a proof from a preliminary version of ~\cite{CE}. We are given an almost-complex manifold $(U,J)$ of complex dimension $n$, and an exhausting Morse function $\varphi$ on $U$ whose critical points all have index $\le n$. We inductively assume that for a certain regular value $a$, the almost-complex structure $J$ restricts to an actual complex structure on the compact submanifold $P=\varphi^{-1}[0,a]$. We also assume that on $P$, the function $\varphi$ is suitably compatible with $J$, i.e., it is a {\em (strictly) plurisubharmonic} (or {\em J-convex}) function. (The significance of this will be indicated later.) We wish to homotope $J$ relative to $P$ to replicate the induction hypotheses for a new regular value $b>a$, where for simplicity we assume $\varphi^{-1}[a,b]$ contains a unique critical point, whose index is $k$. Induction then homotopes $J$ to an honest complex structure on $U$ for which $\varphi$ is an exhausting plurisubharmonic function, and the latter condition guarantees the complex structure is Stein by a theorem of Grauert \cite{Gt}.

We now extract the bundle-theoretic data relevant for the induction step. We ascend to the level of the new regular value $b$ by attaching a handle $h= D^k\times D^{2n-k}$ to $P$ (along with a collar that does not affect the homotopy-theoretic data). At the level of homotopy theory, $h$ is given by its core $D=D^k\times\{0\}$, which is attached along its attaching sphere $S=\partial D$. The complete attaching map $f:\partial_-h= \partial D^k \times D^{2n-k}\to M= \partial P=\varphi^{-1}(a)$ is then determined up to isotopy by by its 1-jet along $S$. On the real hypersurface $M\subset U$, the almost-complex structure $J$ induces a unique field $\xi=TM\cap JTM$ of complex hyperplanes. Since the complex line bundle $(TU|M)/\xi$ is canonically trivialized by a normal vector field to $M$, the complex $(n-1)$-plane field $\xi\subset TM$ carries all of the bundle-theoretic information about $J$ along $M$. The handle $h$ also inherits an almost-complex structure $J_h$ from its inclusion in $U$, and along the attaching region $\partial_-h$, this induces a complex hyperplane field $\xi_h$ (mapping to $\xi$ complex linearly under $f$). The structure $J_h$ is determined up to homotopy rel $S$ by its restriction to the bundle $Th|D$. Thus, the essential data associated to gluing the almost-complex handle are $J_h$ on $Th|D$ and the embedding $f|S$, with the latter covered by a complex bundle map $\xi_h\to\xi$. (Note that this bundle map also carries the normal data determining $f$.)

To complete the proof, we must adjust our setup so that $J_h$ takes a standard form. The descending disk of any critical point of a plurisubharmonic Morse function is {\em totally real}, that is, its tangent bundle contains no complex lines. Thus, it is natural to define the standard complex $k$-handle to be a tubular neighborhood in $\C^n$ of $D^k\subset \R^k \subset \R^n \subset \C^n$, with the standard complex structure $i$. (This also clarifies why the indices of such critical points never exceed $n$.) The space of positively oriented complex vector space structures on $\R^{2n}$ is connected (since a complex basis for any such $J$, completed to a real basis using its $J$-image, is homotopic through real bases to the standard basis), so we may assume $J_h=i$ near the center point of $D$. Since $h$ is contractible, we can then find a homotopy $J_t$ on $h$ from $J_h$ to $i$.

\begin{lem}\label{hprin}
Unless $k=n=2$, there is an isotopy $f_t:\partial_-h\to M$ of the attaching map $f$, and a choice of homotopy $J_t$ on $h$ as above, such that for each $t$, the hyperplane field $\xi_t$ induced by $J_t$ on $\partial_-h$ is sent $J_t$-complex linearly onto $\xi$ by $df_t|S$.
\end{lem}

\noindent We can use $f_t$ to adjust the handle $h$ in $U$, and extend $J_t$ to a homotopy on $U$ rel $P$, deforming our entire setup rel $P$ so that $J_h=i$ along $D$ (for $n\ne2$). By a further adjustment, we can then arrange $h$ to be the standard complex $k$-handle, glued to $P$ holomorphically on a neighborhood of $S$ in $h$. (Here we do not require that $\partial_-h$ maps into $M$, only that its 1-jet along $S$ maps to that of $M$. This is because $M$ is $J$-convex while $\partial_-h$ is $J$-flat. Holomorphicity is then achieved by taking $f$ to be real-analytic on $\R^n$ near $S$, then complexifying.) Now $P\cup h$ is a complex manifold, and $D$ is a totally real and real-analytic submanifold intersecting $M$ standardly (along $S$ we have $JTD\subset TM$). With some additional hard work, Eliashberg produces a regular neighborhood $\hat P$ of such a subset $P\cup D\subset P\cup h$, for which the boundary is $J$-convex. This means that after an isotopy of $U$ fixing $P$ and sending $\partial \hat P$ to $\varphi^{-1}(b)$, we can assume that $\varphi|\hat P$ is plurisubharmonic. The induction step and proof are now completed by a proof of the lemma, which we discuss below.
\end{proof}

\begin{proof}[Proof of Theorem~\ref{ambin}]
Now we are given $U\subset X$, with a suitable Morse function on $U$, and $n\ne 2$. We apply the previous method. For each handle $h\subset U$, we obtain an isotopy $f_t:h\to U$ of the inclusion map (extending the isotopy of $\partial_-h$ given by Lemma~\ref{hprin}), and a homotopy $J_t$ of $J$ rel $P$ on $U$, so that $f_1^*J_1=i$ on $D$. This time, however, we are not allowed to homotope $J$ to construct our Stein open subset of $X$. Instead, we observe that the homotopy $f_t^*J$ agrees with $f_t^*J_t$ on $\xi_t|S$ and for $t=0$. We infer that on $Th|D$, the structure $f_1^*J$ is homotopic rel $\xi_1|S$ to $f_1^*J_1=i$. In particular, there is no homotopy-theoretic obstruction to making $D$ totally real by an isotopy of $f_1$ rel $S$. We can now invoke an h-principle of Gromov to conclude that such an isotopy exists, so we can assume $D$ is totally real and intersects $M$ standardly. After another ($C^1$-small) isotopy, we can assume $D$ is also real-analytic, so Eliashberg's work extends the plurisubharmonic function as before.
\end{proof}

\begin{proof}[Proof of Lemma~\ref{hprin}]
We again sketch the proof from \cite{CE}; details appear in \cite{FoS} Lemma 3.1. We would like to derive the lemma as a sort of h-principle. Since $M$ is a level set of a plurisubharmonic function, it follows that $\xi$ is a {\em contact structure} on $M$. Since the tangent spaces of $S$ lie in $\R^k\subset \C^n$, the fibers of $iTS$ are parallel to $i\R^k$, and so lie tangent to $\partial_-h$. Thus, $TS$ lies in the $i$-complex hyperplane field $\xi_1$, and so the embedding $f_1$ we are trying to construct must send $S$ to a manifold tangent to $\xi$ everywhere. Such a manifold is called {\em Legendrian} if $k=n$ and {\em isotropic} in general. The lemma now amounts to isotoping $f|S$ to an isotropic embedding while suitably controlling data about the almost-complex structures. By combining h-principles of Gromov (making a submanifold of a contact manifold isotropic) and Hirsch (for the homotopy and remaining data), one can at least find a regular {\em homotopy} $f_t$ with the right properties (for all $k,n$). If $k<n$, this will generically be an isotopy, and the lemma follows. When $k=n$, we can still assume $f_1$ is an embedding, but there will generically be finitely many values of $t$ for which $f_t$ is only an immersion. It now suffices to assume $k=n$ and to turn the given regular homotopy into an isotopy for $n\ne 2$.

Since the lemma is trivial for $n<2$, we turn to the next simplest case $n=2$, where the method is only partially successful. There is an obstruction in $\zed$ to turning the regular homotopy $f_t$ into an isotopy rel $t=0,1$, namely the signed number of double points in the surface $F(I\times S)\subset I\times M$, where $F(t,x)=(t,f_t(x))$. (We must change $F$ by a regular homotopy in order to preserve the bundle data.) We wish to kill the obstruction by appending an additional regular homotopy of the Legendrian embedding $f_1$. In order not to lose the Legendrian condition, we homotope through Legendrian immersions. Since every contact structure can be locally identified with the standard model in $\R^{2n-1}$, we can identify a neighborhood of some point of $f_1(S)$ in $M$ with a standard neighborhood in $\R^3$, so that the standard ``front'' projection to $\R^2$ appears as a cusp, as in Figure~\ref{yasha}(a).
\begin{figure}
\centerline{\epsfbox{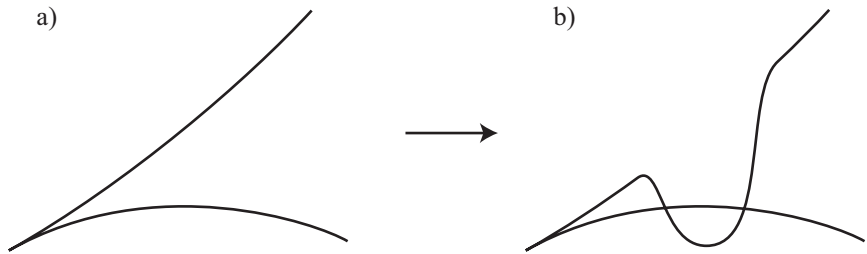}}
\caption{Changing the obstruction to attaching a Stein handle.}
\label{yasha}
\end{figure}
 (See e.g. \cite{GS}, \cite{OS} for more on the standard contact structure and front projections of Legendrian curves.) A regular homotopy of the front projection produces (b) of the figure. This homotopy lifts uniquely to a regular homotopy of Legendrian curves in $\R^3$. (Since negative slopes are always in front of positive slopes, (b) uniquely represents an embedded Legendrian curve in $\R^3$ obtained from (a) by a regular homotopy involving a Legendrian curve with a generic double point.) This local model shows how to change the obstruction by one unit, but unfortunately we cannot control the sign of this change (since the sign of the double point is predetermined through the underlying contact geometry). In fact, it is not possible to arbitrarily change the obstruction without losing control of the topological knot type of $f(S)$, and this is why Eliashberg's method fails for 4-dimensional 2-handles. We return to this phenomenon below.

For $n>2$, the argument works much better. We have the analogous obstruction (in $\zed$ or $\zed_2$, depending on the parity of $n$), but this time we can change it arbitrarily. The analogous front projection is obtained from Figure~\ref{yasha}(a) by crossing with extra dimensions (to obtain an $(n-1)$-manifold in $\R^n$). Instead of pushing an arc of the manifold through the adjacent sheet, we push through a compact $(n-1)$-manifold $Q$ with boundary. One can compute that the resulting change in the obstruction is just the Euler characteristic of $Q$, which can be any integer (whereas a compact 1-manifold must have nonnegative Euler characteristic). Thus, we can modify the regular homotopy to make the obstruction vanish when $n>2$. Now we can perturb $F$ to make its image embedded, by a version of the Whitney trick. (If $M$ is not simply connected, we use the fact that the homotopy is $C^0$-small, so we can assume the Whitney circles are sufficiently small that they are nullhomotopic.) The resulting isotopy satisfies the lemma (and exists unless $k=n=2$).
\end{proof}

 Before extending the method to the 4-dimensional case, Theorem~\ref{abs2}, we clarify terminology. A 4-dimensional compact handlebody with all indices $\le 1$ admits a unique (tight) contact structure on its boundary \cite{ball}. A Legendrian knot in a contact 3-manifold has an induced {\em contact framing}, determined by a vector field in $\xi$ transverse to the knot. In the special case of the attaching circle $S$ in the standard complex 2-handle in $\C^2$, this vector field is just $i\tau$ for a tangent vector field $\tau$ to $S$ in $\R^2$. Since $\tau$ has winding number 1 in $\R^2$, it follows that the contact framing of $S$ differs by one twist from the normal framing to $D$ in the handle. We conclude that for $n=2$, any holomorphic 2-handle must be attached along a Legendrian knot with framing obtained from its contact framing by adding a (left) twist. This is precisely the condition given in Theorem~\ref{abs2}.

\begin{proof}[Proof of Theorem~\ref{abs2}]
When $n=2$, the proof of Theorem~\ref{absn} only fails for 2-handles. The given condition on the framings allows us to bypass Lemma~\ref{hprin}, to attach standard complex 2-handles with $f$ preserving $\xi$ along $S$. These can be arranged to preserve the Stein condition as before.
\end{proof}

\begin{proof}[Proof of Theorem~\ref{ambi2}]
In the ambient case, the main idea for the higher-dimensional Theorem~\ref{ambin} still works, although the details are more delicate. We must incorporate the extra data about the Stein structure homotopic to $J|U$, in order to bypass Lemma~\ref{hprin} for the 2-handles. However, it is too much to expect that the resulting Stein open subset will be holomorphically equivalent to the hypothesized Stein surface, so we must give up some information, relying on the contact structures to carry enough data. For a complete proof, see \cite{steindiff}.
\end{proof}

To understand the context into which Theorem~\ref{abs2} fits, it is helpful to analyze Lemma~\ref{hprin} when $n=2$. In this case, the obstruction can be interpreted as a mismatch between framings. Note that in Figure~\ref{yasha}, for example, we can return from (b) to (a) by an isotopy (although not through Legendrian knots). The resulting regular homotopy from (a) to (b) and back carries along a normal framing, but returns it with two extra twists. (Draw a curve parallel to the one in (a), carry the pair through the regular homotopy, and observe directly that two left twists appear.) The resulting framing change can be interpreted as a normal Euler class, for example on the immersed torus $F(I\times S)$ in $S^1\times M$ that is homologous to $S^1\times f(S)$. The phenomenon now reduces to the observation that for a closed, orientable, generically immersed surface, its normal Euler number differs from the self-intersection number of its homology class by twice the signed number of double points (which follows easily from the definitions, e.g. \cite{GS} Exercise~6.1.1(a) and solution). Thus, the framing condition in Theorem~\ref{abs2} corresponds to suitably killing the obstruction in Lemma~\ref{hprin}. While we cannot apply the Whitney trick as in that lemma, it remains true that attaching circles can be $C^0$-small isotoped to become Legendrian (although we lose control of the almost-complex structures on the corresponding handles). We can then change the obstruction arbitrarily in one direction but not in the other. (Figure~\ref{yasha} changes the obstruction by 2, but it can also be changed by 1 by adding a zig-zag.) To apply Theorem~\ref{abs2}, the problem then becomes to draw a handle diagram for $U$, where the attaching circles are Legendrian with sufficiently large contact framings (or {\em Thurston-Bennequin invariants}, in the language of Legendrian knot theory). This is much harder than applying the higher-dimensional theory, but sometimes still tractable. See e.g. \cite{Ann}, \cite{GS} or \cite{OS} for examples and applications. Once a Legendrian handle diagram is specified for $U$, the homotopy class of the corresponding Stein structure is determined by the {\em rotation numbers} of the Legendrian knots. These are defined as relative Chern classes but can be computed combinatorially in a diagram. A knot can be made Legendrian for only finitely many choices of the rotation number, once the Thurston-Bennequin invariant is specified.  This reflects the fact that a 4-manifold may be Stein realizing some, but not all, choices of the homotopy class of the almost-complex structure.

\begin{proof}[Proof of Theorems~\ref{abstop} and \ref{ambitop}]
In either case, the difficulty is that while we can attach the 2-handles to Legendrian knots with the correct rotation numbers, we cannot do so with the framings required for a Stein structure. (That is, the Thurston-Bennequin invariants will be too negative, so that the contact framings of the Legendrian knots will be far from the framings of the 2-handles.) As we have seen, however, these obstructions can be interpreted as normal Euler numbers, which suitably count double points. We can kill the obstructions by replacing the core disks of the 2-handles by immersed disks. (In the case of Theorem~\ref{ambitop}, the cores may only be topologically embedded, in which case it takes work of Freedman and Quinn to find suitable smoothly immersed disks.) We now obtain a Stein structure, but on the wrong manifold --- each double point of an immersed disk contributes a new generator to the fundamental group. We can kill these generators by adding 2-handles, recovering $U$ but losing the Stein structure. If we instead allow double points in the new 2-handles, we recover the Stein structure but again have extra fundamental group. We inductively kill the extra generators at each stage by adding immersed disks at the next stage, obtaining an infinite union that inherits a Stein structure. This procedure replaces each 2-handle by an infinite tower called a {\em Casson handle}. These Casson handles are simply connected (since any loop lies in a finite subtower, so it is nullhomotopic at the next stage). In fact, according to Freedman~\cite{F}, every Casson handle is homeomorphic to an open 2-handle $D^2\times\R^2$, so our final Stein surface is homeomorphic to $U$. For further details, see \cite{Ann} and \cite{JSG}, respectively.
\end{proof}

\section{Stein handlebodies and pseudoconvex complexes}\label{Onion}

We have seen that for some of our purposes, the Morse and handlebody viewpoints are interchangeable, but for other purposes, one viewpoint is more natural than the other. For 4-manifolds with infinite topology, for example, Theorem~\ref{abs2} is most easily stated using handlebodies. In this section, we adapt the theorems of Section~\ref{Ambient} to the setting of embedded handlebodies, which allows us to use a sharper notion of isotopy (namely {\em ambient} isotopy). We also introduce a third viewpoint, focusing on the totally real CW-complex forming the core of the handlebody. We show that suitably embedded CW-complexes can be isotoped to become totally real and pseudoconvex in the sense of being nested intersections of Stein open subsets. This is particularly striking in the 4-dimensional topological setting, where the resulting uncountable Stein neighborhood systems frequently realize uncountably many diffeomorphism types, even though the neighborhoods are all homeomorphic relative to the embedded complex.

Throughout this section, we rely on informal terminology, hypothesizing that Stein manifolds are made by a particular method. When we construct a Stein manifold, it is clear what method we have used, but if the manifold is given to us in some other way, it is not clear whether the hypothesis applies. In \cite{steindiff} and \cite{steintop} we use more precise and general terminology, with definitions that are currently too cumbersome to state here. Instead, we try to indicate the most important points. For example, a Stein manifold constructed by Eliashberg's method is assumed to be exhibited as the interior of a handlebody with pseudoconvex boundary, and in the ambient setting, this handlebody is embedded. Thus, we are given a partial collaring at infinity by pseudoconvex hypersurfaces, extending to a smooth mapping cylinder structure surrounding a CW-complex core that we can assume is smooth and totally real. The subhandlebodies inherit similar structure.

A smooth (or topological) {\em ambient isotopy} of a manifold $X$ is an isotopy of the identity on $X$ through diffeomorphisms (or homeomorphisms) of $X$. An isotopy of an open subset $U$ of $X$ need not extend to an ambient isotopy of $X$, because the {\em end} of $U$ (the part outside large compact subsets of $U$) can be troublesome. For example, $\C-\{0\}\subset \C$ is isotopic, but not ambiently isotopic, to the complement of a disk, and an arbitrary $X$ is not ambiently isotopic to any proper subset of itself. Thus, Theorem~\ref{ambitop} implies $\CP^1\times \C$ is topologically isotopic inside itself to a Stein surface, but not ambiently. We fix this problem by replacing open subsets by properly embedded handlebodies, which are closed subsets. We obtain the following analog of the main theorems of Section~\ref{Ambient}.

\begin{thm} \label{handle}
Let $H$ be a smooth, $2n$-dimensional handlebody with all handles of index $\le n$.
\begin{itemize}
\item[a)] If $n\ne 2$, then every smooth embedding of $H$ into a complex $n$-manifold $X$ is smoothly isotopic (ambiently if the embedding is proper) to one for which the image has Stein interior and a totally real core. This image is made by Eliashberg's method from the given handle decomposition.
\item[b)] If $n=2$, the same holds for embeddings for which the pulled-back almost-complex structure is homotopic to a Stein structure on $\int H$ made from the given handle decomposition.
\item[c)] If $n=2$, then every topological embedding is topologically isotopic to one whose image has Stein interior and a core that, except for one point on each 2-handle core, is smooth and totally real. If the embedding is proper and the image of $\partial H$ is flat (as defined below), then the isotopy is ambient.
\end{itemize}
\end{thm}

\begin{proof} For (a) and (b), the discussion of Sections~\ref{Ambient} and \ref{proofs} applies with little change, so in particular, (b) is proved in \cite{steindiff} and (a) follows from \cite{E} or \cite{FoS}. One merely needs an additional glance at the proof of Theorem~\ref{ambin} in Section~\ref{proofs} to see that at each stage of the induction, the constructed isotopy can be assumed to be ambient. (Since each handle of $H$ has a smooth, compact outer boundary, the Isotopy Extension Theorem applies.) These isotopies have a well-defined limit on $H$, as well as on $X$ if the embedding is proper, since each point in the domain has a neighborhood on which the isotopies all agree after finitely many stages. (Each map of the limiting isotopy on $X$ is surjective since it is proper. This follows since $H$ is Morse-ordered, so only finitely many handles are stacked above a given handle, and since each map sends each subhandlebody of $H$ into itself.)

The proof of (c) is given in \cite{steintop}. The main new difficulty is that the end of a Casson handle is pathological --- the point-set boundary of a typically embedded Casson handle cannot be a manifold. To address this, we interleave many layers of embedded surfaces between the layers of immersed disks, to obtain the more subtle infinite towers of \cite{FQ}, whose point-set boundaries are 3-manifolds. (Increasing the genus of a surface has the same effect as introducing positive double points when we construct Stein surfaces.) Core disks that are smooth except at one point are a byproduct of Freedman's methods, and in our setting, the smooth parts are also totally real. The flatness hypothesis means that the embedding of $\partial H$ can be extended to an embedding of $\R \times \partial H$. This rules out wild behavior analogous to Alexander's Horned Sphere, and for proper embeddings, it is sufficient to guarantee that the isotopy is ambient. (In \cite{steintop}, properness is replaced by a weaker condition, allowing us to also consider embeddings of infinite handlebodies into regions with compact closure.)
\end{proof}

Since the Stein manifolds of (a) and (b) above are constructed by Eliashberg's method, the resulting boundaries are pseudoconvex hypersurfaces in $X$. (Here we use ``boundary" in the sense of manifolds, which only corresponds to the point-set sense if the embedding is proper.) In contrast, the boundaries obtained in (c) are only topological (flat) 3-manifolds. These will be compact if $H$ is, but may not be smoothable in $X$ by any topological isotopy. It seems reasonable to introduce a notion of ``topological pseudoconvexity'' to describe such 3-manifolds bounding Stein open subsets.

Theorem~\ref{handle} shows that being made by Eliashberg's method is a weak hypothesis --- in fact, every Stein manifold satisfies this hypothesis after a smooth isotopy within itself:

\begin{cor} \label{structure}
Let $U$ be a Stein manifold. Then $U$ is smoothly isotopic within itself to a Stein open subset that is the interior of a smoothly and properly embedded handlebody $H\subset U$ constructed by Eliashberg's method. In particular, the core of $H$ is a smooth, totally real $n$-complex ($n=\dim_\C U$). This can be assumed to be fixed by the given isotopy. Furthermore, $U-\int H$ is diffeomorphic to $[0,1) \times \partial H$.
\end{cor}

\begin{proof} By Lemma~\ref{decomp} (or \ref{steindec} if $n=2$), there is a diffeomorphism $\varphi: \int H'\to U$ from a handlebody interior. Shrink $H'$ into its own interior by the obvious proper isotopy, and apply Theorem~\ref{handle} to the image in $U$ of the resulting $H''$. The consequent ambient isotopy changes $\varphi$ into a new diffeomorphism for which the image $H$ of $H''$ has the required properties.
\end{proof}

We can reinterpret Theorem~\ref{handle} to focus on the cores of the handlebodies. Call an embedded CW-complex {\em tame\/} (in either the topological or smooth setting) if we can thicken it to a handlebody, i.e., it is the core of some embedded handlebody whose handles correspond to cells in the obvious way. (For example, the Fox-Artin Arc in $\R^3$ is not a tame CW-complex, although it could be considered a smooth 1-complex with two 0-cells and an embedded open 1-cell.) It is natural to ask which embedded complexes (up to isotopy) can be thickened to Stein handlebodies as above. We answer the question and observe that the resulting complexes are pseudoconvex in the sense of being nested intersections of Stein neighborhoods.

\begin{cor} \label{ncplx}
Let $K$ be a tamely embedded $n$-complex in a complex $n$-manifold $X$.
\begin{itemize}
\item[a)] In the smooth setting, if $n\ne 2$, then after a $C^0$-small smooth isotopy (smoothly ambient if $K$ is a closed subset, and topologically ambient in general), $K$ is totally real, and is the core of a smoothly embedded handlebody constructed by Eliashberg's method. Thus, $K$ (after isotopy) can be described as a nested intersection of Stein neighborhoods, smoothly isotopic to each other $\rel K$ (ambiently as before), forming a neighborhood system of $K$ if the latter is compact.
\item[b)] If $n=2$, the same holds provided that $K$ thickens to a smoothly embedded handlebody satisfying the hypotheses of Theorem~\ref{abs2}, for which the abstract Stein structure of that theorem is homotopic to the almost-complex structure inherited from the embedding in $X$.
\item[c)] If $n=2$ and $K$ is tamely topologically embedded, then there is a $C^0$-small topological ambient isotopy after which, except on a finite subset of each 2-cell, $K$ is smooth and totally real, and $K$ is the core of a topologically embedded handlebody with Stein interior. In fact, $K$ is a nested intersection of such Stein neighborhoods, topologically ambiently isotopic to each other $\rel K$, forming a neighborhood system of $K$ if the latter is compact.
\end{itemize}
\end{cor}

\noindent Note that in (c), $K$ may not be topologically ambiently isotopic to any smoothly tame complex, and even if it is, there will typically be obstructions to making it totally real. While a smooth 2-cell can always be made totally real except on a finite set, the remaining complex points may obstruct the existence of Stein neighborhoods. For example, $\CP^1\times\{0\}\subset \CP^1\times\C$ is not smoothly isotopic to any complex with a Stein neighborhood, but the theorem gives a topological ambient isotopy to a topological sphere with a Stein neighborhood system.

\begin{proof} Thicken $K$ to a handlebody and apply Theorem~\ref{handle}. In the smooth setting, Eliashberg's method allows the isotopy to be $C^0$-small (once the handles are chosen sufficiently thin). To arrange this in (c), first subdivide $K$ sufficiently, and use the fact that the isotopy of Theorem~\ref{handle} sends each subhandlebody into itself. In either case, we can take the isotopy to be $\epsilon$-small, where $\epsilon$ varies continuously. If $\epsilon=0$ outside a neighborhood within which $K$ is closed, the isotopy will be topologically ambient.  For (a) and (b), the neighborhood system is obtained by repeatedly applying Eliashberg's construction to thicken the given totally real $K$. The neighborhood system in (c) follows from Theorem~\ref{onion} below.
\end{proof}

Eliashberg's method should actually provide more than just a totally real core with a Stein neighborhood system. Every smooth handlebody $H$ whose handles all have index $<\dim H$ admits a smooth mapping cylinder structure, so a map $\psi: [0,1]\times \partial H\to H$ that is the identity on $\{ 1\}\times \partial H$, sends $\{ 0\}\times \partial H$ onto the core $K$ of $H$, and restricts to a diffeomorphism $(0,1]\times \partial H\to H-K$. For each $\sigma\in (0,1]$, let $U_\sigma=\psi( [0,\sigma)\times \partial H)$. When $H$ is compact, these form a neighborhood system of $K$ --- i.e., every neighborhood of $K$ contains some $U_\sigma$.  By choosing Eliashberg's handles to lie in 1-parameter nested families intersecting only in the core, it should be possible to control the construction to inductively arrange the resulting Stein manifold $U_1=\int H$ so that every $U_\sigma$ is also Stein. We would then have a nested family of Stein neighborhoods of the totally real $K$, parametrized by $(0,1]$, with any two neighborhoods smoothly isotopic rel $K$ (ambiently in $H$ for $\sigma<1$). While this has not yet been been worked out in detail, it may be discussed further in \cite{steindiff} or \cite{steintop}. In the 4-dimensional topological setting, we can no longer expect such nice structure, since $\int H-K$ will no longer inherit a smooth product structure. (Even when $H$ is compact, the Stein surface homeomorphic to $\int H$ will not usually be diffeomorphic to the interior of a compact manifold.) Surprisingly, however, we can still use Freedman theory to construct a topological mapping cylinder structure for which $U_\sigma$ is Stein whenever $\sigma$ lies in the standard Cantor set $\Sigma$. While these Stein neighborhoods are still topologically isotopic rel $K$, they will typically not be diffeomorphic to each other.

\begin{thm} \label{onion} \cite{steintop}.
Let $h:\int H\to U$ be a homeomorphism from a handlebody interior to a Stein surface. Suppose either that $U$ was constructed by Eliashberg's method (abstractly or ambiently) as a handlebody interior, exhibited by the diffeomorphism $h$, or that $U$ was obtained by Theorem~\ref{abstop}, \ref{ambitop} or \ref{handle}(c), with $h$ the given homeomorphism. Then after a topological ambient isotopy, the open subsets $U_\sigma\subset U$ coming from the mapping cylinder structure on $H$ are Stein whenever $\sigma\ \in \Sigma$, and except for one point on each 2-cell, the core $K$ is smooth and totally real. If $H_2(U;\zed)$ is nonzero but finitely generated (for example), then we can assume the Stein surfaces $U_\sigma$ for $\sigma\ \in \Sigma$ realize infinitely many diffeomorphism types, and if $U$ topologically embeds in a (possibly infinite) blowup of $\C^2$ (for example), they realize uncountably many diffeomorphism types.
\end{thm}

\noindent A stronger and more detailed version of this is given in \cite{steintop}, with a more general and precise  hypothesis replacing that of the method of construction. (Basically, $U$ should be constructed by Eliashberg's method, allowing some 2-handles of $H$ to be replaced by generalized Casson handles, so that the resulting subhandlebodies with finite towers attached have pseudoconvex boundaries and totally real cores.) Recall that by Corollary~\ref{structure}, every Stein surface is smoothly isotopic within itself to one admitting a diffeomorphism $h$ as above. In fact, the method of that corollary shows that the conclusion of Theorem~\ref{onion} applies to {\em any\/} Stein surface, once we restrict $\sigma$ to $\Sigma \cap [0,\frac12]$. However, the theorem is most striking in the context of Theorems~\ref{abstop}, \ref{ambitop} and \ref{handle}(c). In those cases, when $H$ is a finite handlebody, the resulting $K$ is compact with neighborhood system $\{U_\sigma\}$, even though $U$ and $U_\sigma$ frequently cannot be diffeomorphic to the interior of a finite handlebody. (Consider $H=S^2\times D^2$, for example.) One would then expect the subsets $U_\sigma$ to nearly always be pairwise nondiffeomorphic, and this can be proven in various cases \cite{steintop}.

\section{Appendix: From Morse functions to infinite handlebodies} \label{Appendix}

Our definition of handlebodies in Section~\ref{Abstract} required them to be Morse-ordered and locally finite, without collars between layers of handles. We saw that this was the correct language for expressing Eliashberg's work in dimension 4 (Theorem~\ref{abs2}). Passing to the cores of the handles gave a correspondence between handlebodies and tamely embedded CW-complexes, suggesting Corollary~\ref{ncplx}. Theorem~\ref{onion} required our narrow definition of handlebodies in order to have the underlying mapping cylinder structure. To extend this theorem to arbitrary Stein surfaces we then had to find Stein handlebodies (narrowly defined) inside them via Corollary~\ref{structure}. We now prove the necessary lemmas showing that Morse functions can be refined (introducing critical points) to yield handle decompositions in our narrow sense, even if we start with infinitely many critical points.

\begin{lem}\label{decomp}
Let $U$ be a smooth $m$-manifold with an exhausting Morse function $\varphi$. Then $U$ is diffeomorphic to the interior of a handlebody $H$ with the same maximal index $k$ as $\varphi$.
\end{lem}

\begin{proof} Assume $\varphi$ has infinitely many critical points, since the finite case is well-known. After a slight perturbation, $\varphi$ determines a decomposition $U= \bigcup_{i=0}^\infty P_i$, where $P_0$ is empty and $P_{i+1} = P_i \cup C_i \cup h_i$ for a collar $C_i \approx [0,1] \times \partial P_i$ and handle $h_i$ of index $\le k$. Any handle decomposition of $P_i$ can be extended over $C_i$ using an arbitrary handle decomposition of $\partial P_i$: Each $l$-handle $h$ of $\partial P_i$ contributes an $l$-handle $[\frac12,1] \times h$ and an $(l+1)$-handle $[0,\frac12] \times h$. A perturbation of the product structure on $C_i$ guarantees that each of its handles attaches to handles of lower index as required, as does the additional handle $h_i$ of $P_{i+1}$. Thus, we inductively obtain a handle decomposition $H^*$ of $U$ (containing many $m$-handles) that is compatible with the collars. Let $H \subset H^*$ be the subhandlebody consisting of all handles of index $\le k$ except for $k$-handles of the form $[\frac12,1] \times h$ along each $\partial P_i$. Then for each $i$, $H\cap C_i=[0,1]\times Q_i$, where $Q_i\subset \partial P_i$ is the subhandlebody of handles with index $<k$. It is now routine to construct a diffeomorphism from $\int H$ to $U$: First inductively construct, for each $i$, a diffeomorphism from $\int H$ to $\int (H \cup P_{i+1})$, by pushing the unwanted handles of $C_i$ out of $P_{i+1}$ using the collar structure. Then check that these diffeomorphisms can be assumed to converge to the required diffeomorphism $\int H \to U$.
\end{proof}

\begin{lem}\label{steindec}
Every Stein surface is diffeomorphic to the interior of a handlebody whose handles all have index $\le 2$, and for which each 2-handle is attached along a Legendrian knot in the relevant contact 3-manifold, with framing obtained from the contact framing by adding one left twist. The Stein structure induced by this handle decomposition via Theorem~\ref{abs2} is homotopic (through almost-complex structures) to the original.
\end{lem}

\noindent As we have seen, each 2-handle is attached to a finite subhandlebody with all indices $\le 1$, whose boundary, with its canonical contact structure, is the relevant contact 3-manifold in the lemma.

\begin{proof} We wish to apply the method of the previous proof, but extra care is required for controlling 2-handle framings. We begin with an observation of Goodman \cite{noah}, that every (closed) contact 3-manifold $M$ admits a {\em contact cell decomposition\/}. Using convex surface theory (\cite{cnvx}, \cite{H}, see also \cite{OS}), we construct a contact cell decomposition whose 1-skeleton contains any preassigned Legendrian circle $C$ in $M$. We then show that the 1-skeleton $K$ of such a decomposition is {\em sufficiently fine\/}, by which we mean that every 1-complex in $M$ can be contact isotoped, fixing $K$, into an arbitrarily small neighborhood of $K$. For simplicity, we assume $M$ is tight. Then a tame cell decomposition of $M$ is {\em contact\/} if its 1-cells are Legendrian and its 2-cells are convex with twisting $-1$. Such a decomposition can be constructed, for example, from a triangulation, whose 1-skeleton we can assume contains $C$, by isotoping the 1-cells rel $C$ to be Legendrian with negative twisting. (We measure the twisting using adjacent 2-cells, and arrange negative twisting on the 1-cells of $C$ using a suitably twisted model triangulation near $C$.) The 2-cells can then be assumed to be convex triangles with Legendrian boundaries and negative twisting on each edge, so by the Legendrian Realization Principle (\cite{K}, \cite{H}, see also \cite{OS}), we can subdivide to obtain the required structure. To see that such a 1-skeleton is sufficiently fine, we first contact isotope a given 1-complex $K'$ into a preassigned neighborhood $W$ of the 2-skeleton. Choose a 3-ball in the interior of each 3-cell, with convex boundary contained in $W$. By convex surface theory, the boundary is essentially uniquely determined, so by \cite{ball}, each ball can be contactomorphically identified with the standard ball in $\R^3$. Using Gray's Theorem \cite{Gray}, it is now easy to push $K'$ off of each ball into $W$ by a contact isotopy fixing $K$. For any neighborhood $V$ of the 1-skeleton, we can choose $W$ sufficiently thin that $K'$ can then be pushed into $V$ as required. This follows from examining a standard neighborhood of each 2-cell, where, for example, we push first toward the two elliptic boundary points, then transversely, preserving the characteristic foliation away from the elliptic points in each case.

Now let $(U,J)$ be an arbitrary Stein surface. The basic theory of Stein manifolds \cite{CE} guarantees an exhausting plurisubharmonic Morse function $\varphi$ on $U$, that we can assume is 1-1 on its critical set. Its critical points all have index $\le \dim_\C U=2$, and its regular level sets inherit contact structures. Each descending disk of an index-2 critical point intersects nearby level sets below it in a Legendrian circle, and when we thicken the disk to obtain a 2-handle, it is attached with framing differing from the contact framing by a left twist, as desired for the lemma. As before, the lemma is easy and well-known if $\varphi$ has only finitely many critical points, so we again restrict to the infinite case and construct the required handle structure from $\varphi$ using induction. For our induction hypotheses, we assume there is a handlebody $H$, whose handles all have index $\le 2$ and satisfy the required framing condition, with a diffeomorphism $f: H\to P=\varphi^{-1}[0,a]$ for a preassigned regular value $a$. By Theorem~\ref{abs2} and its proof (in the direction independent of Lemma~\ref{steindec}), there is an induced Stein structure $J_H$ on $H$ and contact structure on $\partial H$. We inductively assume we are given a neighborhood $V$ in $\partial H$ of a sufficiently fine 1-complex $K_0$, whose $l$-cells ($l=0,1$) are disjoint from the handles of $H$ with index $>l$. (In comparison with the previous proof, $V$ corresponds to the intersection of $\partial P_i$ with the handles of $H^*$ of index $<k$, containing the foundation $Q_i$ for the next layer of handles.) We also assume that $f|\partial H$ is a contactomorphism to $M=\partial P$ with contact structure $\xi$ induced by $J$, and that on $P$ minus finitely many interior points, $f_*J_H$ agrees with $J$ up to homotopy through almost-complex structures inducing the given $\xi$. (It seems unlikely that deleting points is actually necessary, but it is easier to avoid the issue.) Note that the space of complex structures on $TU|M$ inducing a given plane field on $TM$ is contractible, so $J$ is essentially determined on $M$ by $\xi$. The induction hypotheses are trivially satisfied for $a=-1$. For a regular value $b>a$ such that $[a,b]$ contains a unique critical value, $\varphi^{-1}[0,b]$ is obtained from $P$ by attaching a collar and a handle; we wish to extend over these to reproduce the induction hypotheses with $b$ in place of $a$.

We begin the induction step by adding handles to $H$ along $V$, to form part of a collar of $\partial H$ as in the proof of Lemma~\ref{decomp}. Let $D$ denote the descending disk of the unique critical point in $\varphi^{-1}[a,b]$, with $C=\partial D$ in $M=\varphi^{-1}(a)$ a Legendrian circle or two points or empty, depending on the index of the critical point. Let $K\supset C$ be the 1-skeleton of a contact cell decomposition of $M$. Since the previously given complex $K_0\subset \partial H$ is sufficiently fine, there is a contact isotopy of $M$ sending $K$ into $f(V)$. By Gray's Theorem, we can assume the isotopy pushes the 0-cells of $K$ sufficiently close to those of $f(K_0)$ that they avoid the 1- and 2-handles of $H$. By reversing the isotopy and extending it to an isotopy of $f$ supported near $\partial H$, we can adjust $f$ so that $f(V)$ contains $K$ (as originally located, so in particular $C=\partial D\subset f(V)$), without disturbing the induction hypotheses. By construction, the 1-complex $f^{-1}(K)\subset V\subset \partial H$ thickens to a handlebody $Q\subset V$ whose $l$-handles ($l=0,1)$ avoid the handles of $H$ with index $>l$. (Compare with $Q_i$ in the previous proof.) To extend $H$ over $[0,1]\times (\mbox{\rm 0-handles})\subset [0,1]\times Q$, let $H'$ be the result of attaching a canceling 0-handle and 1-handle to $H$ at each vertex of $f^{-1}(K)$, by Eliashberg's method. Then $H'$ is a (Morse-ordered) handlebody, with a Stein structure $J_{H'}$ extending $J_H$, and a contact boundary. The latter agrees with $(\partial H,f^*\xi)$ outside a collection of 3-balls. Since a tight contact 3-ball is determined by its boundary \cite{ball}, there is a contactomorphism $\partial H'\to\partial H$. This extends to a diffeomorphism $f':H'\to H$ that is supported on a disjoint union of 4-balls, so after we remove the center points of the 4-balls, $f'_*J_{H'}$ is homotopic to $J_H$ preserving $f^*\xi$ on $\partial H$. We replace $H$ by $H'$ and $f$ by $f\circ f'$, preserving the induction hypotheses (ignoring $V$). We next construct the rest of $[0,1]\times Q$. For each 1-cell of $K$, the preimage in $\partial H'$ is a Legendrian arc with endpoints in $\{1\}\times Q$. After trimming the ends of this arc by removing its intersection with $(\frac34,1]\times Q$, connect its new endpoints by a 1-handle, then add a canceling 2-handle along the resulting Legendrian circle. If we add the handles by Eliashberg's method, the new 2-handles will satisfy the required framing condition. The resulting handlebody $H''$ differs from $H'$ only on a collection of 4-balls, so the previous argument allows us to replace $H$ by $H''$ without disturbing the induction hypotheses (again ignoring $V$). The preimage $C''$ of $C$ in $\partial H''$ sits in $\{1\}\times Q$ on top of the partial collar, and is a Legendrian circle if $\dim C=1$.

To complete the induction step, we must attach a new handle along $C''$, suitably map the result onto $\varphi^{-1}[0,b]$, and reconstruct $K_0\subset V$. Let $\hat P\subset U$ be obtained from $P$ by attaching a handle with core $D$ using Eliashberg's method. Let $\hat H$ be obtained from $H''$ by attaching a handle $h$ with the corresponding attaching data, so we get a diffeomorphism $\hat f:\hat H\to \hat P$. (For a 2-handle, the attaching data consist of the Legendrian circle $C''$ and a normal framing, which by construction satisfies the condition required for the lemma.) The two new handles intersect $\partial \hat H$ and $\partial \hat P$, respectively, in regions $T_i$ that are either a solid torus or a punctured 3-sphere, depending on the index of $h$. In each case, the (tight) contact structure is determined on $T_i$ by its restriction to $\partial T_i$. (For punctured 3-spheres, this follows from \cite{ball}. In the remaining case, the 2-handle attaches to a neighborhood of a Legendrian circle, whose boundary can be assumed convex with a 2-component dividing set $\Gamma$. The framing condition guarantees that each component of $\Gamma$ intersects the meridian of $T_i$ in a single point, so the classification of contact structures on a solid torus uniquely determines the contact structure \cite{Gi}, \cite{H}, cf. also Theorem~5.1.30 of \cite{OS}.) Thus, we can assume $\hat f$ restricts to a contactomorphism $\partial\hat H\to \partial\hat P$. (Any self-diffeomorphism of $S^3$ extends over $B^4$.) Since the new handles are 4-balls, we recover a homotopy from $\hat f_*J_{\hat H}$ to $J|\hat P$ in the complement of a finite set, preserving the contact boundary as usual. By a result in \cite{CE}, we can assume $\hat P$ was constructed so that there is a new exhausting plurisubharmonic function $\hat \varphi$ on $U$ for which $\partial\hat P$ and $\varphi^{-1}(b)$ are level sets, with no critical points in between. We can then construct an isotopy $f_t$ of $\hat f$ such that $f_1(\partial\hat H)=\varphi^{-1}(b)$ and each $f_t(\partial\hat H)$ is a level set of $\hat\varphi$, hence a contact manifold. By Gray's Theorem, we can assume each $f_t|\partial\hat H$ is a contactomorphism. Then $f_1:\hat H\to \varphi^{-1}[0,b]$ satisfies the induction hypotheses at the new level, once we construct a new $K_0$ and $V$. For this last step, recall that $K$ is the 1-skeleton of a contact cell decomposition of $M$. We pull this contactomorphically back to $\partial H''$, and modify it inside a neighborhood of $C''$ in $\{1\}\times Q$ to obtain a cell decomposition of $\partial \hat H$. This only requires care when $h$ has index 2. In that case, we require a pushoff of $\partial \partial_-h$ in $\partial H''$ to be a subcomplex of the new structure, and fill in $\partial \hat H\cap h$ with a single 2-cell and 3-cell. Since the 2-cell intersects the dividing set of $\partial \partial_-h$ in two points, it has twisting $-1$. Now for any index, we make the decomposition contact as in the first paragraph, noting that no change occurs outside a neighborhood of $C''$, and that no subdivision is necessary in $\partial h$ in the index-2 case. The new sufficiently fine 1-skeleton $\hat K_0$ lies in $\hat V=\partial\hat H\cap((\{1\}\times Q)\cup h)$, it avoids the 2-handles of $\hat H$, and after a contact isotopy supported in $\hat V$, its 0-cells avoid the 1-handles of $\hat H$. Thus, $\hat K_0$ and $\int \hat V$ complete the induction hypotheses.

We have now inductively constructed a nested sequence of handlebodies, each with an embedding into $U$. Since the handles added at the $i^{th}$ stage are attached along a neighborhood $V$ contained in a subhandlebody added at the $(i-1)^{st}$ stage, the union $H_\infty$ is locally finite, and hence an infinite handlebody. Each embedding into $U$ agrees with the previous one on the domain of the latter, except on a neighborhood of its boundary. It is routine to arrange these embeddings so that each point of $\int H_\infty$ has a neighborhood on which all but finitely many embeddings agree. We then have a limiting map $f_\infty:\int H_\infty\to U$ that is easily seen to be a diffeomorphism. (Compare with the diffeomorphism constructed in the previous proof.) Similarly, the homotopies of almost-complex structures yield a limiting homotopy from $(f_\infty)_*J_\infty$ to $J$ on the complement of a discrete and closed subset. (Here $J_\infty$ denotes the Stein structure inherited by $\int H_\infty$ since its 2-handles are suitably framed.) Since the (2-dimensional) core of $H_\infty$ can be chosen to miss the singular set, it is easy to construct a new homotopy that is defined everywhere on $U$.
\end{proof}



\begin{thebibliography}{MM}

\bibitem[BG]{BG}
\v Z. Bi\v zaca and R. Gompf,
{\em Elliptic surfaces and some simple exotic $\real^4$'s},
J. Diff. Geom.
{\bf 43} (1996), 458--504.

\bibitem[C]{C}
A. Casson,
{\em Three lectures on new infinite constructions in 4-dimensional manifolds},
(notes prepared by L.~Guillou),
A la Recherche de la Topologie Perdue,
Progress in Mathematics
vol. 62,
Birkh\"auser, 1986, pp.~201--244.

\bibitem[CE]{CE}
K. Cieliebak and Y. Eliashberg,
{\em Symplectic Geometry of Stein manifolds},
book in preparation


\bibitem[DF]{DF}
S. DeMichelis and M. Freedman,
{\em Uncountably many exotic $\real^4$'s in standard 4-space},
J. Diff. Geom. {\bf 35} (1992), 219--254.

\bibitem[E1]{E}
Y. Eliashberg,
{\em Topological characterization of Stein manifolds of dimension $>2$},
Int. J. of Math. {\bf 1} (1990),  29--46.

\bibitem[E2]{ball}
Y. Eliashberg,
{\em Contact 3-manifolds twenty years since J. Martinet's work},
Ann. Inst. Fourier {\bf 42} (1992),  165--192.

\bibitem[EM]{EM}
Y. Eliashberg and N. Mishachev,
{\em Introduction to the h-Principle},
Grad. Studies in Math. {\bf 48},
Amer. Math. Soc., Providence, 2002.

\bibitem[FoS]{FoS}
F. Forstneri{\v c} and M. Slapar,
{\em Stein structures and holomorphic mappings},
Math. Zeit. {\bf 256} (2007), 615--646.


\bibitem[F]{F}
M. Freedman,
{\em The topology of four-dimensional manifolds},
J. Diff. Geom. {\bf 17} (1982),  357--453.

\bibitem[FQ]{FQ}
M. Freedman and F. Quinn,
{\em Topology of 4-Manifolds},
Princeton Math. Ser., 39,
Princeton Univ. Press,
Princeton, NJ, 1990.

\bibitem[Gi1]{cnvx}
E. Giroux,
{\em Convexit\'e en topologie de contact},
Comment. Math. Helv. {\bf 66} (1991), 637--677.

\bibitem[Gi2]{Gi}
E. Giroux,
{\em Structures de contact sur les vari\'et\'es fibr\'ees en cercles d'une surface},
Comment. Math. Helv. {\bf 76} (2001), 218--262.

\bibitem[G1]{Ann}
R. Gompf,
{\em Handlebody construction of Stein surfaces},
Ann. Math. {\bf148} (1998), 619--693.

\bibitem[G2]{JSG}
R. Gompf
{\em Stein surfaces as open subsets of $\complex^2$},
J. Symplectic Geom. {\bf3} (2005), 565-587.

\bibitem[G3]{steindiff}
R. Gompf
{\em Smooth embeddings with Stein surface images},
in preparation.

\bibitem[G4]{steintop}
R. Gompf
{\em Creating Stein surfaces by topological isotopy},
in preparation.

\bibitem[GS]{GS}
R. Gompf and A. Stipscz,
{\em 4-manifolds and Kirby calculus},
Grad. Studies in Math. {\bf 20},
Amer. Math. Soc., Providence, 1999.

\bibitem[Go]{noah}
N. Goodman
{\em Contact structures and open books},
dissertation, Univ. of Texas at Austin, 2003.

\bibitem[Gt]{Gt}
H. Grauert,
{\em On Levi's problem},
Ann. Math. {\bf68} (1958), 460--472.

\bibitem[Gr]{Gray}
J. W. Gray,
{\em Some global properties of contact structures},
Ann. Math. {\bf69} (1959), 421--450.

\bibitem[Gv]{Gr}
M. Gromov,
{\em Partial Differential Relations},
Springer-Verlag, 1986.

\bibitem[H]{H}
K. Honda,
{\em On the classification of tight contact structures, I},
Geom. Topol. {\bf4} (2000), 309--368.

\bibitem[K]{K}
Y. Kanda,
{\em The classification of tight contact structures on the 3-torus},
Comm. Anal. Geom. {\bf5} (1997), 413--438.

\bibitem[LM]{LM}
P. Lisca and G. Mati\'c,
{\em Tight contact structures and Seiberg-Witten invariants},
Invent. Math. {\bf 129} (1997), 509--525.

\bibitem[M]{M}
J. Milnor,
{\em Morse Theory},
Ann. Math. Studies 51,
Princeton Univ. Press, 1963.

\bibitem[OS]{OS}
B. Ozbagci and A. Stipsicz,
{\em Surgery on contact 3-manifolds and Stein surfaces},
Bolyai Society Math Studies {\bf 13},
Springer, 2004.

\bibitem[Ta]{Ta}
C. Taubes,
{\em Gauge theory on asymptotically periodic 4-manifolds},
J. Diff. Geom. {\bf 25} (1987),  363--430.

\bibitem[T]{T}
L. Taylor,
{\em An invariant of smooth 4-manifolds},
Geom. and Top. {\bf1} (1997), 71--89.


\end{thebibliography}
\end{document}